\documentclass{amsart}

\usepackage[pagebackref]{hyperref}
\usepackage{cite}
\usepackage[all]{xy}
\usepackage{enumerate}
\usepackage{mathrsfs}
\usepackage{physics}
\usepackage{color}
\usepackage{microtype} 

\newtheorem{theorem}{Theorem}[section]
\newtheorem{lemma}[theorem]{Lemma}
\newtheorem{corollary}[theorem]{Corollary}

\theoremstyle{definition}

\theoremstyle{remark}
\newtheorem{remark}[theorem]{Remark}

\numberwithin{equation}{section}

\def\M{\overline{M}}
\def\g{\overline{g}}
\def\c{\overline{\nabla }}
\def\X{\boldsymbol{X}}
\def\H{\boldsymbol{H}}
\def\L{\mathcal{L}}
\def\B{\boldsymbol{B}}
\def\R{\overline{R}}
\def\Ric{\mathrm{Ric}}
\def\Hess{\mathrm{Hess\,}}
\def\K{\overline{K}}
\def\x{\boldsymbol{x}}
\def\T{\boldsymbol{T}}

\allowdisplaybreaks[4]

\begin{document}

\title[Gradient estimates on conformal solitons]{Gradient estimates for positive eigenfunctions of $\L$-operator on conformal solitons and their applications}


\author[G. Zhao]{Guangwen Zhao}
\address{Department of Mathematics, School of Sciences, Wuhan University of Technology, Wuhan 430070, China}
\curraddr{}
\email{gwzhao@whut.edu.cn}
\thanks{This work is partially supported by the National Natural Science Foundation of China (12001410).}





\dedicatory{}

\begin{abstract}
  We prove a local gradient estimate for positive eigenfunctions of $\L$-operator on conformal solitons given by a general conformal vector field. As an application, we obtain a Liouville type theorem for $\L u=0$, which improves the one of Li--Sun (Acta Math. Sin. (Engl. Ser.), 37(11):1768–1782, 2021.). We also consider applications where manifolds are special conformal solitons. Especially in the case of self-shrinkers, a better Liouville type theorem is obtained.
\end{abstract}

\maketitle

\tableofcontents

\section{Introduction} 
\label{sec:introduction}

Let $M^n$ be a smooth manifold of dimension $n$ and $(\M^{n+p},\g)$ a smooth Riemannian manifold of dimension $n+p$. A family of immersions
\[ 
  F_t:M\to \M
\]
is called a mean curvature flow of $M$ in $\M$ if it satisfies
\[
  \pdv{t}F_t=\H_t,
\]
where $\H_t$ is the mean curvature vector of $F_t(M)$ in $\M$.

Following K. Smoczyk \cite{Smoczyk2001}, we call $M$ is a \textit{conformal soliton} to the mean curvature flow if 
\begin{equation}\label{eq-a}
  \H=\X^N
\end{equation}
for a conformal vector field $\X$, where $\H$ is the mean curvature vector of $M$ in $\M$ and $\X^N$ is the normal projection of $\X$ to the normal bundle of $M$ in $\M$. Recall that a smooth vector field $\X$ on $\M$ is called \textit{conformal} if in local coordinate, $\c_j\X_i+\c_i\X_j=\lambda \g_{ij}$ for some smooth function $\lambda $ on $\M$. In other word, a conformal vector field $\X$ is a smooth vector field which satisfies 
\begin{equation}\label{eq-b}
  L_{\X}\g=\lambda \g
\end{equation}
for some smooth function $\lambda $ on $\M$, where $L$ is the Lie derivative.

Let $F:M\to (\M,\g)$ be a conformal soliton. We denote by $g=F^*\g$ the induced metric on $M$, which makes $F:(M,g)\to (\M,\g)$ is a isometric immersion. Let $\nabla $ and $\c$ be the Levi--Civita connections on $M$ and $\M$, respectively. There is an important elliptic operator on $M$ as follows:
\begin{equation}\label{eq-c}
  \L\star =\Delta\star +\langle \X,\nabla \star \rangle_{\g},
\end{equation}
where $\Delta $ and $\nabla $ is the Laplacian and the gradient operator on $M$. Through this paper, we call this elliptic operator the $\L$-operator for convenience. In\cite{Li2021}, X. Li and J. Sun considered the special conformal vector field $\X$ satisfies
\[
  \c_j\X_i=\lambda \g_{ij},
\]
which can be written as $\X=\c \omega $ for some smooth function $\omega $ according to Poincar\'e lemma, as long as $\M$ is simply connected. They derived gradient estimates for the positive solutions of $\L u=0$ and $\L u=\pdv{u}{t}$. In their setting, the $\L$-operator can be written as $\L\star =\Delta \star +\langle \c \omega,\nabla \star \rangle_{\g}$, where $\c$ is the gradient operator on $\M$.

In this paper, we focus on the general case and consider conformal solitons given by a general conformal vector field of the form \eqref{eq-b}, and study gradient estimates of the eigenfunctions of $\L$-operator. In precise, we will study gradient estimates for positive solution of the linear elliptic equation
\begin{equation}\label{eq-h}
  \L u=-\mu u,
\end{equation}
where $\mu$ is a constant. Throughout this paper we use $V=\X^T$ to represent the tangential projection of $\X$ to the tangent bundle of $M$, and accordingly the $\L$-operator can be viewed the $V$-Laplacian
\begin{equation*}
  \L\star =\Delta\star +\langle \X^T,\nabla \star \rangle_{\g}=\Delta \star +\langle V,\nabla \star \rangle_g=:\Delta_V\star ,
\end{equation*}
which is a generalized Laplacian introduced by Q. Chen, J. Jost and G. Wang \cite{Chen2015}. As in \cite{Chen2012,Wei2009}, the Bakry--\'Emery Ricci tensor are defined by
\[
  \Ric_V=\Ric -\frac{1}{2}L_Vg,
\]
where $\Ric $ is the Ricci tensor, and $L$ is the Lie derivative. 

In the sequel, we say that $M$ is a conformal soliton in $(\M,\g)$, meaning that $M$ is a conformal soliton given by a conformal vector field of the form \eqref{eq-b} on $\M$. Our first main result is the following:
\begin{theorem}[Gradient estimate]\label{thm-a}
  Let $M^n$ be a conformal soliton in $(\M,\g)$. Suppose that $\K\ge -K_0$, $|\B|\le B_0$, $|\X^T|\le V_0$ and $\lambda \le \lambda_0$, where $K_0,B_0$ and $V_0$ are nonnegative constants. Let $u$ be a positive solution to \eqref{eq-h}. Then there exists a positive constant $C$ depending on $n,\mu , a, K_0, B_0, V_0$ and $\lambda_0$, such taht
  \begin{equation}\label{eq-w}
    \sup_{B_{a/2}(p)}\frac{|\nabla u|}{u}\le C.
  \end{equation}

  More explicitly, for any $\alpha >0$, we have
  \begin{equation}\label{eq-v}
    \sup_{B_{a/2}(p)}\frac{|\nabla u|^2}{u^2}\le \widetilde{C}\left(\Lambda_0+|\Lambda_0|+|\mu |+\frac{\sqrt{K_0}+B_0+V_0}{a}+\frac{1}{a^2}\right),
  \end{equation}
  where $\widetilde{C}$ is a positive constant depending on $n$ and $\alpha $, and
  \begin{align*}
    \Lambda_0=(n-1)K_0+B_0^2+\frac{V_0^2}{\alpha (n-1)}+\frac{\lambda_0}{2}.
  \end{align*}
\end{theorem}

By letting $\mu =0$ and taking $a\to \infty $ in estimate \eqref{eq-v}, we immediately get the following Liouville type theorem for $\L u=0$.
\begin{corollary}[Liouville type theorem]\label{cor-a}
  Let $M^n$ be a complete noncompact conformal soliton in $(\M,\g)$. Suppose that $\K\ge -K_0$, $|\B|\le B_0$, $|\X^T|\le V_0$ and $\lambda \le \lambda_0$, where $K_0,B_0$ and $V_0$ are nonnegative constants. If there is a positive number $\alpha $ such that $(n-1)K_0+B_0^2+\frac{V_0^2}{\alpha (n-1)}+\frac{\lambda_0}{2}\le 0$ , then any positive or bounded solution to $\L u=0$ must be constant.
\end{corollary}

\begin{remark}
  Our Corollary~\ref{cor-a} improves the corresponding result obtained by X. Li and J. Sun in \cite{Li2021}. In fact, they obtained a Liouville type theorem for positive or bounded solutions of $\L u=0$ on a conformal soliton given by a special conformal vector field of the form $\c_j\X_i=\lambda \g_{ij}$, under the assumptions of $\K\ge -K_0,\ |\B|\le c,\ |\nabla \omega |\le k,\ \overline{\Hess}\omega \le m\g $, and $(n-1)K_0+c^2+\frac{k^2}{\alpha (n-1)}+m= 0$, where $\omega $ is a smooth function satisfies $\c \omega =\X$.
\end{remark}

As another application of Theorem~\ref{thm-a}, we can also prove a Harnack type inequality for positive eigenfunctions of $\L$-operator.
\begin{corollary}[Harnack type inequality]\label{cor-b}
  The conditions are the same as in Theorem~\ref{thm-a}. Then there exists a positive constant $\widehat{C}$ depending on $n,\mu , a, K_0, B_0, V_0$ and $\lambda_0$, such that
  \[
  \sup_{B_{a/2}(p)}u\le \widehat{C}\inf_{B_{a/2}(p)}u.
  \]
\end{corollary}

Finally, we consider conformal solitons given by some special conformal vector fields in Euclidean space. When the ambient manifold $(\M,\g)$ is the Euclidean space $(\mathbb{R}^{n+1},\delta )$, it is easy to know that the position vector $\x$ and any constant vector $\T$ are conformal. Conformal solitons defined by $\x$ and $\T$ are self-shrinkers and translating solitons, respectively. 

In the case $\X=-\frac{1}{2}\x$, a simple calculation yields $\L_{\X}\delta =-\frac{1}{2}L_{\x}\delta =-\delta $, which is $\lambda \equiv -1$ in \eqref{eq-b}. It follows from Theorem~\ref{thm-a} that
\begin{corollary}
  Let $\x:M^n\to \mathbb{R}^{n+1}$ be a complete self-shrinker, which satisfies $\H=-\frac{1}{2}\x^N$. Suppose that $|\B|\le B_0$ and $|\x^T|\le v_0$, where $B_0$ and $v_0$ are nonnegative constants. Let $u$ be a positive solution to $\Delta u-\frac{1}{2}\langle \x,\nabla u\rangle_\delta =-\mu u$. Then for any $\alpha >0$, there exists a positive constant $C_1$ depending on $n$ and $\alpha $, such that
  \[
  \sup_{B_{a/2}(p)}\frac{|\nabla u|^2}{u^2}\le C_1\left(\max\left\{B_0^2+\frac{v_0^2}{4\alpha (n-1)}-\frac{1}{2},0\right\}+|\mu |+\frac{B_0+v_0}{a}+\frac{1}{a^2}\right).
  \]
\end{corollary}
We also have the following Liouville type theorem, which can be found in \cite{Zhu2018}, with the slight difference that our conditions here are more relaxed.
\begin{corollary}
  Let $\x:M^n\to \mathbb{R}^{n+1}$ be a complete noncompact self-shrinker, which satisfies $\H=-\frac{1}{2}\x^N$. Suppose that $|\B|\le B_0$ and $|\x^T|\le v_0$, where $B_0$ and $v_0$ are nonnegative constants. If there is a positive number $\alpha $ such that $B_0^2+\frac{v_0^2}{4\alpha (n-1)}\le \frac{1}{2}$, then any positive or bounded solution to $\Delta u-\frac{1}{2}\langle \x,\nabla u\rangle_\delta =0$ must be constant.
\end{corollary}

In the case $\X=\T$, we have $L_{\X}\delta =0$ and then $\lambda \equiv 0$ in \eqref{eq-b}. It follows from Theorem~\ref{thm-a} that
\begin{corollary}
  Let $\x :M\to \mathbb{R}^{n+1}$ be a complete translating soliton, which satisfies $\H=T^N$. Suppose that $|\B|\le B_0$ and $|\T^T|\le T_0$. Let $u$ be a positive solution to $\Delta u+\langle \T,\nabla u\rangle_\delta =-\mu u$. Then for any $\alpha >0$, there exists a positive constant $C_2$ depending on $n$ and $\alpha $, such that
  \[
  \sup_{B_{a/2}(p)}\frac{|\nabla u|^2}{u^2}\le C_2\left(B_0^2+\frac{T_0^2}{\alpha (n-1)}+|\mu |+\frac{B_0+T_0}{a}+\frac{1}{a^2}\right).
  \]
\end{corollary}


\section{Preliminaries} 
\label{sec:preliminaries}

Let $(M^n,g)$ be a isometric immersed submanifold of $(\M^{n+p},\g)$, and $\nabla $ and $\c$ be their Levi--Civita connections, respectively. Here and in the sequel, we equate a vector field on $M$ with its image in $\M$. For any $X,Y\in \Gamma(TM)$, we have $\nabla_XY=\left(\c_XY\right)^T$, and define the second fundamental form $\B$ of $M$ by $\B(X,Y)=\left(\c_XY\right)^N=\c_XY-\nabla_XY$, which is a symmetric bilinear form on $TM$ with values in $NM$. The mean curvature vector $\H$ of $M$ is defined by $\H=\mathrm{tr}_g\B=\sum_{i=1}^n\B(e_i,e_i)$, where $\{e_i\}_{i=1}^n$ is a local orthonormal frame field of $M$.

Now we compare the Lie derivatives of $g$ and $\g$: for $X,Y\in \Gamma(TM)$, recall that we specify $V=\X^T$, and we have  
\begin{equation}\label{eq-d}
  \begin{split}
    L_Vg(X,Y)
    =&V\langle X,Y\rangle_g -\langle [V,X],Y\rangle_g -\langle X,[V,Y]\rangle_g\\
    =&\langle \nabla_XV,Y\rangle_g+\langle X,\nabla_YV\rangle_g\\
    =&X\langle V,Y\rangle_g-\langle V,\nabla_XY\rangle_g+Y\langle X,V\rangle_g-\langle \nabla_YX,V\rangle_g\\
    =&X\langle \X,Y\rangle_{\g}-\langle \X,\c_XY\rangle_{\g}+\langle \H,\B(X,Y)\rangle_{\g}\\
    &+Y\langle Y,\X\rangle_{\g}-\langle \c_YX,\X\rangle_{\g}+\langle \B(Y,X),\H\rangle_{\g}\\
    =&\langle \c_X\X,Y\rangle_{\g}+\langle X,\c_YX\rangle_{\g}+2\langle \X,\B(X,Y)\rangle_{\g}\\
    =&L_{\X}\g(X,Y)+2\langle \H,\B(X,Y)\rangle_{\g}.
  \end{split}
\end{equation} 
Denote the curvature tensor of $M$ by $R$, which is defined in this paper as $R(X,Y)Z=\nabla_X\nabla_YZ-\nabla_Y\nabla_XZ-\nabla_{[X,Y]}Z$ for $X,Y,Z\in \Gamma(TM)$. Similarly, denote by $\R$ the curvature tensor of $\M$. We have the Gauss equation
\begin{equation*}
  \langle R(X,Y)Z,W\rangle_g=\langle \R(X,Y)Z,W\rangle_{\g}-\langle \B(X,Z),\B(Y,W)\rangle_{\g}+\langle \B(X,W),\B(Y,Z)\rangle_{\g}.
\end{equation*}
Choosing a local orthonormal frame field $\{e_i\}_{i=1}^n$ of $M$, let $Y=Z=e_i$ and sum both sides of the above equation, we obtain
\begin{equation}\label{eq-e}
  \Ric(X,Y)=\sum_{i=1}^n\langle \R(X,e_i)e_i,Y\rangle_{\g}-\sum_{i=1}^n\langle \B(X,e_i),\B(Y,e_i)\rangle_{\g}+\langle \B(X,Y),\H\rangle_{\g},
\end{equation}
where we replaced $W$ with $Y$. From \eqref{eq-d} and \eqref{eq-e}, the Bakry--\'Emery Ricci tensor
\begin{equation}\label{eq-f}
  \begin{split}
    \Ric_V(X,Y)
    =&\Ric(X,Y)-\frac{1}{2}L_Vg(X,Y)\\
    =&\sum_{i=1}^n\langle \R(X,e_i)e_i,Y\rangle_{\g}-\sum_{i=1}^n\langle \B(X,e_i),\B(Y,e_i)\rangle_{\g}-\frac{1}{2}L_{\X}\g(X,Y).
  \end{split}
\end{equation}
So we have 
\begin{align}
  \Ric  &\ge -(n-1)K_0-2B_0^2,\label{eq-k1}\\
  \Ric_V&\ge -\left((n-1)K_0+B_0^2+\frac{\lambda_0 }{2}\right),\label{eq-k}
\end{align}
if the sectional curvature of $\M$ has a lower bound $-K_0$, the norm of the second fundamental form of $M$ has a upper bound $B_0$ and $\lambda \le \lambda_0$ on $\M$. 


\section{Proofs of the main results} 
\label{sec:proof_of_main_results}

To prove Theorem~\ref{thm-a}, we need a lemma. We first list a elementary inequality that will be used in the proof of our lemma. For any $a,b\in \mathbb{R}$ and $\alpha >0$, we have
\begin{equation}\label{eq-g}
  (a+b)^2\ge \frac{a^2}{1+\alpha }-\frac{b^2}{\alpha },
\end{equation}
and equality holds if and only if $b=-\frac{\alpha }{1+\alpha }a$.
\begin{lemma}\label{lem-a}
  The conditions are the same as in Theorem~\ref{thm-a}. Then for $\alpha >0$ and $\beta >0$, we have
  \begin{equation}\label{eq-l}
    \begin{split}
      |\nabla u|\Delta_V|\nabla u|
      \ge &\frac{|\nabla |\nabla u||^2}{(1+\alpha )(1+\beta )(n-1)}-\frac{(\mu u)^2}{\beta (1+\alpha )(n-1)}\\
      &-\left((n-1)K_0+B_0^2+\frac{\lambda_0 }{2}+\mu +\frac{V_0^2}{\alpha (n-1)}\right)|\nabla u|^2.
    \end{split}
  \end{equation}
\end{lemma}
\begin{proof}
  Recall that $\L u=\Delta_Vu$. From the Bochner formula
  \[
  \frac{1}{2}\Delta_V|\nabla u|^2=|\Hess u|^2+\Ric_V(\nabla u,\nabla u)+\langle \nabla u,\nabla \Delta_Vu\rangle_g
  \]
  and the fact 
  \[
  \frac{1}{2}\Delta_V|\nabla u|^2=|\nabla |\nabla u||^2+|\nabla u|\Delta_V|\nabla u|,
  \]
  combined with $u$ is a solution of \eqref{eq-h}, we have
  \begin{equation}\label{eq-i}
    \begin{split}
      |\nabla u|\Delta_V|\nabla u|
      =|\Hess u|^2+\Ric_V(\nabla u,\nabla u)-\mu |\nabla u|^2-|\nabla |\nabla u||^2.
    \end{split}
  \end{equation}
  Proceeding, given $p\in M$, choose the normal coordinates around $p$ so that $u_i(p)=0$ for $2\le i\le n$ and $u_1(p)=|\nabla u|(p)$, and then
  \[
  |\nabla|\nabla u||^2=|\nabla u_1|^2=\sum_{j=1}^nu_{1j}^2.
  \]
  Therefore,
  \begin{align*}
    |\Hess u|^2-|\nabla |\nabla u||^2
    =&\sum_{i,j=1}^nu_{ij}^2-\sum_{j=1}^nu_{1j}^2\\
    =&\sum_{i=2}^n\sum_{j=1}^nu_{ij}^2\\
    \ge &\sum_{i=2}^nu_{i1}^2+\sum_{i=2}^nu_{ii}^2\\
    \ge &\sum_{i=2}^nu_{i1}^2+\frac{1}{n-1}\left(\sum_{i=2}^nu_{ii}\right)^2\\
    =&\sum_{i=2}^nu_{i1}^2+\frac{1}{n-1}(\mu u+u_{11}+V^1u_1).
  \end{align*}
  By using \eqref{eq-g} twice, for any $\alpha >0$ and $\beta >0$, we have
  \begin{align*}
    \mu u+u_{11}+V^1u_1
    \ge &\frac{(\mu u+u_{11})^2}{1+\alpha }-\frac{(V^1u_1)^2}{\alpha }\\
    \ge &\frac{1}{1+\alpha }\left(\frac{u_{11}^2}{1+\beta }-\frac{(\mu u)^2}{\beta }\right)-\frac{(V^1u_1)^2}{\alpha }\\
    =&\frac{u_{11}^2}{(1+\alpha )(1+\beta )}-\frac{(\mu u)^2}{\beta (1+\alpha )}-\frac{1}{\alpha }\langle V,\nabla u\rangle_g^2.
  \end{align*}
  It follows that
  \begin{align*}
    &|\Hess u|^2-|\nabla |\nabla u||^2\\
    \ge &\sum_{i=2}^nu_{i1}^2+\frac{u_{11}^2}{(1+\alpha )(1+\beta )(n-1)}-\frac{(\mu u)^2}{\beta (1+\alpha )(n-1)}-\frac{1}{\alpha (n-1)}\langle V,\nabla u\rangle_g^2\\
    \ge &\frac{|\nabla |\nabla u||^2}{(1+\alpha )(1+\beta )(n-1)}-\frac{(\mu u)^2}{\beta (1+\alpha )(n-1)}-\frac{|V|^2}{\alpha (n-1)}|\nabla u|^2
  \end{align*}
  Substitute this inequality and \eqref{eq-k} into \eqref{eq-i}, combine the conditions of the lemma, and note that $\L=\Delta_V$ and $V=\X^T$, we complete the proof.
\end{proof}

\begin{proof}[Proof of Theorem~\ref{thm-a}]
  We define $\phi :=|\nabla \ln u|=\frac{|\nabla u|}{u}$. Then we have
  \[
  \nabla \phi =\frac{\nabla |\nabla u|}{u}-\frac{|\nabla u|\nabla u}{u^2}.
  \]
  From the define of $\phi $ we also have
  \begin{align*}
    \Delta_V|\nabla u|
    =&\Delta_V(u\phi )\\
    =&u\Delta_V\phi +2\langle \nabla u,\nabla \phi \rangle_g+\phi \Delta_Vu\\
    =&u\Delta_V\phi +2\langle \nabla u,\nabla \phi \rangle_g-\mu |\nabla u|.
  \end{align*}
  By Lemma~\ref{lem-a}, for any $\alpha >0$ and $\beta >0$, at any point where $|\nabla u|\ne 0$,
  \begin{align*}
    \Delta_V\phi =&\frac{\Delta_V|\nabla u|}{u}-\frac{2\langle \nabla u,\nabla \phi \rangle_g}{u}+\mu \phi \\
    \ge &\frac{1}{u|\nabla u|}\bigg \{\frac{|\nabla |\nabla u||^2}{(1+\alpha )(1+\beta )(n-1)}-\frac{(\mu u)^2}{\beta (1+\alpha )(n-1)}\\
    &-\left((n-1)K_0+B_0^2+\frac{\lambda_0 }{2}+\mu +\frac{V_0^2}{\alpha (n-1)}\right)|\nabla u|^2\bigg \}-\frac{2\langle \nabla u,\nabla \phi \rangle_g}{u}+\mu \phi \\
    =&\frac{|\nabla |\nabla u||^2}{(1+\alpha )(1+\beta )(n-1)u|\nabla u|}-\frac{\mu^2}{\beta (1+\alpha )(n-1)\phi }\\
    &-\left((n-1)K_0+B_0^2+\frac{V_0^2}{\alpha (n-1)}+\frac{\lambda_0 }{2}\right)\phi -\frac{2\langle \nabla u,\nabla \phi \rangle_g}{u}
  \end{align*}
  As in \cite[P. 19]{Schoen1994}, we have for any $\epsilon >0$,
  \begin{align*}
    -\frac{2\langle \nabla u,\nabla \phi \rangle_g}{u}
    =&-(2-\epsilon )\frac{\langle \nabla u,\nabla \phi \rangle_g}{u}-\epsilon \frac{\langle \nabla u,\nabla \phi \rangle_g}{u}\\
    =&-(2-\epsilon )\frac{\langle \nabla u,\nabla \phi \rangle_g}{u}-\epsilon \frac{\langle \nabla |\nabla u|,\nabla u\rangle_g}{u^2}+\epsilon \frac{|\nabla u|^3}{u^3}\\
    \ge &-(2-\epsilon )\frac{\langle \nabla u,\nabla \phi \rangle_g}{u}-\epsilon \frac{|\nabla |\nabla u||\cdot |\nabla u|}{u^2}+\epsilon \frac{|\nabla u|^3}{u^3}\\
    \ge &-(2-\epsilon )\frac{\langle \nabla u,\nabla \phi \rangle_g}{u}-\frac{\epsilon }{2}\left(\frac{|\nabla |\nabla u||^2}{u|\nabla u|}+\frac{|\nabla u|^3}{u^3}\right)+\epsilon \frac{|\nabla u|^3}{u^3}\\
    =&-(2-\epsilon )\frac{\langle \nabla u,\nabla \phi \rangle_g}{u}-\frac{\epsilon }{2}\frac{|\nabla |\nabla u||^2}{u|\nabla u|}+\frac{\epsilon }{2}\phi^3.
  \end{align*}
  Taking $\epsilon =\frac{2}{(1+\alpha )(1+\beta )(n-1)}$, we arrive at
  \begin{equation}\label{eq-m}
    \begin{split}
      \Delta_V\phi 
      \ge &-\frac{\mu^2}{\beta (1+\alpha )(n-1)\phi }-\left((n-1)K_0+B_0^2+\frac{V_0^2}{\alpha (n-1)}+\frac{\lambda_0 }{2}\right)\phi \\
      &-\left(2-\frac{2}{(1+\alpha )(1+\beta )(n-1)}\right)\frac{\langle \nabla u,\nabla \phi \rangle_g}{u}+\frac{\phi^3}{(1+\alpha )(1+\beta )(n-1)}.
    \end{split}
  \end{equation}
  Given $a>0$, define a function $F$ in $B_a(p)$ as follows:
  \[
  F(x)=(a^2-r^2(x))\phi(x)=(a^2-r^2)\frac{|\nabla u|}{u},
  \]
  where $r(x)=\mathrm{dist}(x,p)$ is the distance function from $p$. Without loss of generality, assume $|\nabla u|\ne 0$. Since $F|_{\partial B_a(p)}=0$ and $F>0$, it attains its maximum at some point $x_1\in B_a(p)$. The argument of E. Calabi \cite{Calabi1958} allows us to assume that $x_1$ is not a cut point of $p$. Hence $F$ is smooth near $x_1$.

  By the maximum principle, we have at $x_1$,
  \begin{equation}\label{eq-n}
    0=\nabla F=-\phi \nabla (r^2)+(a^2-r^2)\nabla \phi 
  \end{equation}
  and
  \begin{equation}\label{eq-o}
    0\ge \Delta_VF=(a^2-r^2)\Delta_V\phi -\phi \Delta_V(r^2)-2\langle \nabla \phi ,\nabla (r^2)\rangle_g.
  \end{equation}
  From \eqref{eq-n}, we have
  \begin{equation}\label{eq-p}
    \frac{\nabla \phi }{\phi }=\frac{\nabla (r^2)}{a^2-r^2},
  \end{equation}
  and then \eqref{eq-o} reduces to
  \begin{equation}\label{eq-q}
    \begin{split}
      0\ge &\frac{\Delta_V\phi }{\phi }-\frac{\Delta_V(r^2)}{a^2-r^2}-\frac{2|\nabla (r^2)|^2}{(a^2-r^2)^2}\\
      =&\frac{\Delta_V\phi }{\phi }-\frac{\Delta (r^2)}{a^2-r^2}-\frac{\langle V,\nabla(r^2)\rangle_g}{a^2-r^2}-\frac{8r^2}{(a^2-r^2)^2}.
    \end{split}
  \end{equation}
  Since the Ricci curvature of $M$ satisfies \eqref{eq-k1}, the classical Laplacian comparison theorem gives
  \begin{align*}
    \Delta (r^2)
    =&2r\Delta r+2\\
    \le &2r\left(\frac{n-1}{r}+(n-1)\sqrt{K_0+\frac{2B_0^2}{n-1}}\right)+2\\
    \le &2n+2(n-1)\sqrt{K_0+\frac{2B_0^2}{n-1}}a.
  \end{align*}
  Therefore, it follows from \eqref{eq-q} that
  \begin{equation}\label{eq-r}
    \begin{split}
      0\ge \frac{\Delta_V\phi }{\phi }-\frac{1}{a^2-r^2}\left(2n+2(n-1)\sqrt{K_0+\frac{2B_0^2}{n-1}}a+2V_0r\right)-\frac{8r^2}{(a^2-r^2)^2}.
    \end{split}
  \end{equation}
  By combining inequality \eqref{eq-m} and \eqref{eq-r}, we obtain
  \begin{align*}
    0\ge &-\frac{\mu^2}{\beta (1+\alpha )(n-1)}\phi^{-2}-\left((n-1)K_0+B_0^2+\frac{V_0^2}{\alpha (n-1)}+\frac{\lambda_0 }{2}\right)\\
    &-\frac{1}{a^2-r^2}\left(2n+2(n-1)\sqrt{K_0+\frac{2B_0^2}{n-1}}a+2V_0r\right)-\frac{8r^2}{(a^2-r^2)^2}\\
    &-\left(2-\frac{2}{(1+\alpha )(1+\beta )(n-1)}\right)\frac{\langle \nabla u,\nabla \phi \rangle_g}{u\phi }+\frac{\phi^2}{(1+\alpha )(1+\beta )(n-1)}.
  \end{align*}
  By the Cauchy--Schwarz inequality and equality \eqref{eq-p}, we deduce
  \[
  -\frac{\langle \nabla u,\nabla \phi \rangle_g}{u\phi }=-\frac{2r\langle \nabla u,\nabla r\rangle_g}{u(a^2-r^2)}\ge -\frac{2r}{a^2-r^2}\phi .
  \]
  We thus get
  \begin{align*}
    0\ge &-\frac{\mu^2}{\beta (1+\alpha )(n-1)}\phi^{-2}-\left((n-1)K_0+B_0^2+\frac{V_0^2}{\alpha (n-1)}+\frac{\lambda_0 }{2}\right)\\
    &-\frac{1}{a^2-r^2}\left(2n+2(n-1)\sqrt{K_0+\frac{2B_0^2}{n-1}}a+2V_0r\right)-\frac{8r^2}{(a^2-r^2)^2}\\
    &-4\left(1-\frac{1}{(1+\alpha )(1+\beta )(n-1)}\right)\frac{r}{a^2-r^2}\phi +\frac{\phi^2}{(1+\alpha )(1+\beta )(n-1)}.
  \end{align*}
  Multiplying through by $(a^2-r^2)^4\phi^2$, and noticing that $0\le r\le a$, we arrive at
  \begin{equation}\label{eq-s}
    \begin{split}
      0\ge &\frac{1}{(1+\alpha )(1+\beta )(n-1)}F^4\\
      &-4a\left(1-\frac{1}{(1+\alpha )(1+\beta )(n-1)}\right)F^3-C_1F^2-\frac{\mu^2a^8}{\beta (1+\alpha )(n-1)},
    \end{split}
  \end{equation}
  where 
  \begin{align*}
    C_1=&\left((n-1)K_0+B_0^2+\frac{V_0^2}{\alpha (n-1)}+\frac{\lambda_0 }{2}\right)a^4\\
    &+2\left(n+4+(n-1)\sqrt{K_0+\frac{2B_0^2}{n-1}}a+V_0a\right)a^2.
  \end{align*}
  Consider the corresponding polynomial of the right hand side of \eqref{eq-s}
  \begin{align*}
    P(y)=&\frac{1}{(1+\alpha )(1+\beta )(n-1)}y^4\\
      &-4a\left(1-\frac{1}{(1+\alpha )(1+\beta )(n-1)}\right)y^3-C_1y^2-\frac{\mu^2a^8}{\beta (1+\alpha )(n-1)}.
  \end{align*}
  We see that $P(y)$ is a quartic polynomial with a positive leading coefficient and negative constant term, so Vieta's formulas tell us that it has exactly a positive real root and a negative real root. We fix the positive numbers $\alpha $ and $\beta $, there is a positive constant $C_2$ depending on $n,\mu , a, K_0, B_0, V_0$ and $\lambda_0$, such that $F(x_1)=\sup\limits_{B_a(p)}F\le C_2$. Therefore, in $B_{a/2}(p)$, we have 
  \[
  \frac{3a^2}{4}\sup_{B_{a/2}(p)}\frac{|\nabla u|}{u}\le C_2.
  \] 
  This completes the proof of the first part of the theorem.
  
  Now we want to make our estimate more specific. By the AM--GM inequality, we have for any $\delta \in (0,1)$,
  \begin{align*}
    &-4a\left(1-\frac{1}{(1+\alpha )(1+\beta )(n-1)}\right)F^3\\
    \ge &-\frac{\delta }{(1+\alpha )(1+\beta )(n-1)}F^4\\
    &-\frac{4}{\delta }(1+\alpha )(1+\beta )(n-1)\left(1-\frac{1}{(1+\alpha )(1+\beta )(n-1)}\right)^2a^2F^2.
  \end{align*}
  Hence \eqref{eq-s} reduces to
  \begin{equation}\label{eq-t}
    0\ge \frac{1-\delta }{(1+\alpha )(1+\beta )(n-1)}F^4-\Lambda F^2-\frac{\mu^2a^8}{\beta (1+\alpha )(n-1)},
  \end{equation}
  where $\Lambda =\Lambda_0a^4+\Lambda_1a^3+\Lambda_2a^2$ with
  \begin{align*}
    \Lambda_0=&(n-1)K_0+B_0^2+\frac{V_0^2}{\alpha (n-1)}+\frac{\lambda_0}{2},\\
    \Lambda_1=&2(n-1)\sqrt{K_0+\frac{2B_0^2}{n-1}}+2V_0\ge 0,\\
    \Lambda_2=&2n+8+\frac{4}{\delta }(1+\alpha )(1+\beta )(n-1)\left(1-\frac{1}{(1+\alpha )(1+\beta )(n-1)}\right)^2>0.
  \end{align*}
  Notice that $\Lambda $ is not necessarily positive, because $\Lambda_0$ can be negative and its absolute value is large enough. For a positive number $\tilde a$ and a nonnegative number $\tilde c$, from the quadratic inequality $\tilde ay^2-\tilde by-\tilde c\le 0$ we have $y\le \frac{\tilde b+\sqrt{{\tilde b}^2+4\tilde a\tilde c}}{2\tilde a}$. So from inequality \eqref{eq-t} we get
  \begin{equation*}
    \begin{split}
      (F(x_1))^2
      \le &\frac{(1+\alpha )(1+\beta )(n-1)}{2(1-\delta )}\left\{\Lambda +\left(\Lambda ^2+\frac{4(1-\delta )\mu^2a^8}{\beta (1+\beta )(1+\alpha )^2(n-1)^2}\right)^{\frac{1}{2}}\right\}\\
      \le &\frac{(1+\alpha )(1+\beta )(n-1)}{2(1-\delta )}(\Lambda +|\Lambda |)+\sqrt{\frac{1+\beta }{\beta (1-\delta )}}|\mu |a^4,
    \end{split}
  \end{equation*}
  where we used the simple fact that $\sqrt{x^2+y^2}\le |x|+|y|$ in the second inequality. It follows that
  \begin{equation*}
    \begin{split}
      \frac{9a^4}{16}\sup_{B_{a/2}(p)}\frac{|\nabla u|^2}{u^2}
      \le &\frac{(1+\alpha )(1+\beta )(n-1)}{2(1-\delta )}(\Lambda +|\Lambda |)+\sqrt{\frac{1+\beta }{\beta (1-\delta )}}|\mu |a^4\\
      \le &\frac{(1+\alpha )(1+\beta )(n-1)}{2(1-\delta )}\left((\Lambda_0+|\Lambda_0|)a^4+2\Lambda_1a^3+2\Lambda_2a^2\right)\\
      &+\sqrt{\frac{1+\beta }{\beta (1-\delta )}}|\mu |a^4.
    \end{split} 
  \end{equation*}
  Therefore, by fixing $\beta $ and $\delta $, there is a positive constant $\widetilde{C}$ depending on $n$ and $\alpha $, such that \eqref{eq-v} holds.
\end{proof}

\begin{proof}[Proof of Corollary~\ref{cor-b}]
  Choose $x_1,x_2\in B_{a/2}(p)$ such that $u(x_1)=\sup_{B_{a/2}(p)}u$ and $u(x_2)=\inf_{B_{a/2}(p)}u$. Let $\gamma $ be a shortest geodesic connecting $x_1$ and $x_2$, and the triangle inequality implies $\gamma \subset B_a(p)$. In $B_a(p)$, there is also a gradient inequality such as \eqref{eq-w}, and integrating it over $\gamma $ yields
  \begin{align*}
    \int_\gamma \frac{|\nabla u|}{u}\dd s\le C\int_\gamma \dd s\le 2Ca.
  \end{align*}
  It follows that
  \[
  \ln \frac{u(x_1)}{u(x_2)}\le \int_\gamma \frac{|\nabla u|}{u}\dd s\le 2Ca.
  \]
  Hence we obtain $u(x_1)\le \widehat{C}u(x_2)$, where $\widehat{C}=e^{2Ca}$ is a constant depending on $n,\mu , a, K_0, B_0, V_0$ and $\lambda_0$.
\end{proof}



\begin{thebibliography}{1}

\bibitem{Calabi1958}
E.~Calabi.
\newblock An extension of {E}. {H}opf's maximum principle with an application
  to {R}iemannian geometry.
\newblock {\em Duke Math. J.}, 25:45--56, 1958.

\bibitem{Chen2012}
Q.~Chen, J.~Jost, and H.~Qiu.
\newblock Existence and {L}iouville theorems for {$V$}-harmonic maps from
  complete manifolds.
\newblock {\em Ann. Global Anal. Geom.}, 42(4):565--584, 2012.

\bibitem{Chen2015}
Q.~Chen, J.~Jost, and G.~Wang.
\newblock A maximum principle for generalizations of harmonic maps in
  {H}ermitian, affine, {W}eyl, and {F}insler geometry.
\newblock {\em J. Geom. Anal.}, 25(4):2407--2426, 2015.

\bibitem{Li2021}
X.~Li and J.~Sun.
\newblock Gradient estimate for the positive solutions of {$\mathcal Lu = 0$}
  and {$\mathcal Lu =\frac{\partial u}{\partial t}$} on conformal solitons.
\newblock {\em Acta Math. Sin. (Engl. Ser.)}, 37(11):1768--1782, 2021.

\bibitem{Schoen1994}
R.~M. Schoen and S.-T. Yau.
\newblock {\em Lectures on differential geometry}.
\newblock International press, 1994.

\bibitem{Smoczyk2001}
K.~Smoczyk.
\newblock A relation between mean curvature flow solitons and minimal
  submanifolds.
\newblock {\em Math. Nachr.}, 229:175--186, 2001.

\bibitem{Wei2009}
G.~Wei and W.~Wylie.
\newblock Comparison geometry for the {B}akry-{E}mery {R}icci tensor.
\newblock {\em J. Differential Geom.}, 83(2):377--405, 2009.

\bibitem{Zhu2018}
Y.~Zhu and Q.~Chen.
\newblock Gradient estimates for the positive solutions of {$\mathcal {L}u=0$}
  on self-shrinkers.
\newblock {\em Mediterr. J. Math.}, 15(1):Paper No. 28, 14, 2018.

\end{thebibliography}
\end{document}